\documentclass[amssymb, 11pt]{amsart}

\usepackage{amsmath,amssymb,amsfonts,amsthm,latexsym,url,bm,tabls,verbatim}

\theoremstyle{plain}      \newtheorem{theorem}{Theorem}[section]
                          \newtheorem{lemma}[theorem]{Lemma}
                          \newtheorem{prop}[theorem]{Proposition}
                          \newtheorem{cor}[theorem]{Corollary}

\theoremstyle{remark}     \newtheorem*{rem}{Remark}

\theoremstyle{definition}

\usepackage{varioref}
\labelformat{section}{Section~#1} \labelformat{figure}{Figure~#1}
\labelformat{table}{Table~#1}

\numberwithin{equation}{section}

\usepackage[numbers]{natbib}

\date{}

\newcommand\tv{\text{TV}}
\newcommand\real{\mathbb{R}}

\begin{document}

\title{Carries, Shuffling, and Symmetric Functions}

\author{Persi Diaconis}
\address{Department of Mathematics and Statistics\\
Stanford University\\ Stanford, CA 94305}

\author{Jason Fulman}
\address{Department of Mathematics\\
University of Southern California\\
Los Angeles, CA 90089-2532} \email{fulman@usc.edu}



\date{Version of January 27, 2009}

\begin{abstract}
  The ``carries'' when $n$ random numbers are added base $b$ form a
  Markov chain with an ``amazing'' transition matrix determined by
  Holte \cite{holte}. This same Markov chain occurs in following the number of
  descents or rising sequences when $n$ cards are repeatedly riffle
  shuffled. We give generating and symmetric function proofs and
  determine the rate of convergence of this Markov chain to
  stationarity. Similar results are given for type $B$ shuffles. We also
  develop connections with Gaussian autoregressive processes and the
  Veronese mapping of commutative algebra.
\end{abstract}

\maketitle

\section{Introduction}\label{sec1}

We use generating functions and symmetric function theory to explain a
surprising coincidence: when $n$-long integers are added base-$b$, the
distribution of ``carries'' is the same as the distribution of
descents when $n$ cards are repeatedly riffled shuffled. The
explanation yields a sharp analysis of convergence to stationarity of
the associated Markov chains. A similar analysis goes through for
``type $B$'' shuffles. In this introduction, we first explain the
carries process, then riffle shuffling and finally the connection.

\subsection{Carries}\label{sub11}

Consider adding three $50$-digit numbers base 10 (in the top row,
italics are used to indicate the carries):
\begin{small}
\begin{equation*}
\begin{array}{lllllllllll}
\textit{1}&\textit{12021}&\textit{01111}&\textit{11111}&\textit{11111}&
\textit{11011}&\textit{10111}&\textit{01111}&\textit{11111}&\textit{21011}&\textit{1112}\\
 &43935&23749&58561&74916&62215&47448&33196&51990&19807&27075\\
 &48537&53642&77448&32760&14421&72142&82116&37225&43300&51498\\
 &33618&41327&41561&16257&43616&55134&82714&63369&87142&45607\\ \hline
 1 &26091&18719&77571&23934&20253&74725&98027&52585&50250&24180
\end{array}
\end{equation*}\end{small}

For this example, $6/50=12\%$ of the columns have a carry of zero,
$40/50=80\%$ have a carry of one and $4/50=8\%$ have a carry of two.

If $n$ integers (base $b$) are produced by choosing their digits uniformly
at random in $\{0,1,2,\dots,b-1\}$, the sequence of carries
$\kappa_0=0,\kappa_1,\kappa_2,\dots$ forms a Markov chain taking values in
$\{0,1,2,\dots,n-1\}$. Holte \cite{holte} studied this Markov chain and
found fascinating structure in its ``amazing'' transition matrix
$(P(i,j))$. Here $P(i,j)$ is the chance that the next carry is $j$ given
that the last carry was $i$, and he showed, for $0\leq i, j\leq n-1$, that
\begin{equation}
P(i,j)=\frac{1}{b^n}\sum_{l=0}^{j-\lfloor i/b\rfloor} (-1)^l
\binom{n+1}{l} \binom{n-1-i+(j+1-l)b}{n}. \label{11}
\end{equation}
For example, when $n=3$ the matrix becomes
\begin{equation*}
\frac{1}{6b^2}\begin{pmatrix}
b^2+3b+2&4b^2-4&b^2-3b+2\\
b^2-1&4b^2+2&b^2-1\\
b^2-3b+2&4b^2-4&b^2+3b+2
\end{pmatrix}.
\end{equation*}
Among many other things, Holte shows that the $j$th entry of the left
eigenvector with eigenvalue $1$ is $A(n,j)/n!$, with $A(n,j)$ the Eulerian
number: the number of permutations in the symmetric group $S_n$ with $j$
descents. Here $\sigma \in S_n$ is said to have a descent at $i$ if
$\sigma(i+1)<\sigma(i)$. So $\bm{5}\,1\,\bm{3}\,2\,4$ has two descents.
The fundamental theorem of Markov chain theory gives that $A(n,j)/n!$ is
the long term frequency of carries of $j$ when long random numbers are
added. Note that this is independent of the base $b$. When $n=3$,
$A(3,0)/6=1/6$, $A(3,1)/6=2/3$, $A(3,2)/6=1/6$ very roughly matching the
example above. We give alternative derivations of this at the end of
\ref{sec2}.

We will not detail the many nice properties Holte found but warmly
recommend his paper \cite{holte}. Some further properties are in
\cite{bren}, \cite{pd173}, which give appearances of this same matrix in
card shuffling and in the Veronese construction for graded algebras. This
is developed briefly in \ref{sec5} below.

\subsection{Shuffling}\label{sub12}

The usual method of shuffling cards proceeds by cutting a deck of $n$
cards into two approximately equal piles and then riffling the two
piles together into one pile. A realistic mathematical model was
created by Gilbert--Shannon--Reeds: cut off $c$ cards with probability
$\binom{n}{c}/2^n$. Drop cards sequentially as follows: if the left
pile has $A$ cards and the right pile has $B$ cards, drop the next
card from the bottom of the left pile with probability $A/(A+B)$ and
from the right pile with probability $B/(A+B)$. This is continued
until all cards are dropped.

A careful analysis of riffle shuffles is carried out in \cite{pd92} using
a generalization to $b$-shuffles. There, a deck of cards is cut into $b$
packets of size $c_1,c_2,\dots,c_b$ with probability $\binom{n}{c_1\dots
c_b}/b^n$. The packets are riffled together by dropping the next card with
probability proportional to packet size. Thus the original
Gilbert--Shannon--Reeds model corresponds to a $2$-shuffle. Two basic
facts established in \cite{pd92} are:

\begin{itemize}
\item The chance of the permutation $\sigma$ arising after a $b$-shuffle is
\begin{equation} \label{12} \frac{\binom{n+b-d(\sigma^{-1})-1}{n}}{b^n}
\end{equation}
with $d(\sigma^{-1})$ the number of descents in $\sigma^{-1}$.

\item An $a$-shuffle followed by a $b$-shuffle is the same as an
$ab$-shuffle.
\end{itemize}

Thus the result of $r$ $2$-shuffles is the same as a single $2^r$ shuffle
and so formula \eqref{12} gives a closed form expression for the chance of
any permutation after $r$ $2$-shuffles. This and some calculus allow a
sharp analysis of the rate of convergence: roughly $\tfrac32\log_2 n+c$
shuffles suffice to make the distribution within $2^{-c}$ of the uniform
distribution. Further details are in \cite{pd92}.

The combinatorics of riffle shuffles has expanded. An enumerative theory
of cycle and other properties under the $b$-shuffle measure \eqref{12} is
equivalent to the Gessel--Reutenauer enumeration jointly by cycles and
descents \cite{pd77,ges}. The combinatorics of riffle shuffling is
essentially the same as quasi-symmetric function theory
\cite{ful02,sta01}. There are extensions to other types (see \ref{sec4}
below) and to random walk on the chambers of hyperplane arrangements
\cite{bid,brd} and buildings \cite{brown}.  Much of this development is
surveyed in \cite{pd30}. Interesting new developments are in \cite{pd171}.

\subsection{The connection}\label{sub13}

Carries and riffle shuffling seem like different subjects. However, if
$P_b$ denotes the matrix \eqref{11}, Holte \cite{holte} showed that
\begin{equation} \label{hol}
P_aP_b=P_{ab}
\end{equation}
The eigenvalues of the matrix $P_b$ turn out to be the same as the
eigenvalues of the $b$-shuffle transition matrix (the multiplicities are
different). This, and the appearance of descents in both subjects, led us
to suspect and then prove an intimate connection. In \ref{sec2} we prove
the following.
\begin{theorem}
  The chance that the base-b carries chain goes from $0$ to $j$ in $r$ steps
  is equal to the chance that the permutation in $S_n$ obtained by
  performing $r$ successive $b$-shuffles (started at the identity) has
  $j$ descents.
\label{thm1}
\end{theorem}

We give a generating function proof which also yields a similar
statement for the inverse permutation along with enumerative results
of Gessel in \ref{sec2}. We have subsequently found a bijective proof of the theorem
which shows that the transition matrices of carries \eqref{11} and the Markov
chain generated by the number of descents after successive
$b$-shuffles are the same \cite{pd173}.

The more analytic proof given here allows us to use the
Robinson--Schensted--Knuth (RSK) correspondence and symmetric function
theory to show that the number of descents (and in fact any function of
the descent set) after $r$ $2$-shuffles is close to stationarity when
$r=\log_2 n+c$. (Note from \cite{pd92} that $\tfrac32 \log_2 n+c$ are
required for \textit{all} aspects of the permutation to be close to
stationarity.) The correspondence with carries shows that the carries
chain `settles down' after $\log_2 n+c$. Refining this, we show that for
large $n$, $\frac{1}{2}\log_b(n)+c$ steps of the carries chain are
necessary and sufficient for convergence to stationarity. Details are in
\ref{sec3}.

The discussion so far has all been on the permutation group. There are
well-established ``type $B$'' (hyperoctahedral)-shuffles
\cite{pd92,ber,ful01}. In \ref{sec4} we develop a parallel ``carries
process'' and show that theorems about type $B$ shuffles translate into
theorems about adding numbers. We also point out a connection with the
theory of rounding. \ref{sec5} shows that for large $n$, the carries
process is well approximated by a Gaussian autoregressive process, and
develops the connection with the Veronese mapping of commutative algebra.

\section{Two Markov Chains}\label{sec2}

In this section we show that two processes derived from the Markov chain
of repeated $b$-shuffles on the symmetric group are Markov chains with
transition probabilities from $0$ to $j$, the same as the carries chain.
As background, note that usually a function of a Markov chain is
\textit{not} a Markov chain. A simple example is nearest neighbor random
walk on the integers mod $n$, with $n$ odd, $n\geq 7$. Let the walk start
at $0$ and move left or right with probability $1/2$. Let $f(j)=1$ for
$0\leq j\leq(n-1)/2$, $f(j)=-1$ otherwise. If steps of the original walk
are denoted $X_0=0, X_1,X_2,\dots$ and $Y_j=f(X_j)$, then
$\{X_j\}_{j=0}^\infty$ is a Markov chain but $\{Y_j\}_{j=0}^\infty$ is
not: $\mathbb{P} \{Y_3=+|Y_2=+\}=2/3$, $\mathbb{P} \{Y_3=+|Y_2=+,
Y_1=+\}=1$. The literature on conditions for Markovianity are often called
``lumping of Markov chains.'' A useful introduction is \cite{kem} with
\cite{rog} a sophisticated extension.

To begin, we show that the two basic facts about riffle shuffles give a
generating function identity of Gessel (unpublished).
\begin{prop} \label{gen} Let $\sigma$ be a permutation with $d$
  descents. Let $c_{ij}^d$ be the number of ordered pairs $(\tau,\mu)$
  of permutations in $S_n$ such that $\tau$ has $i$ descents, $\mu$
  has $j$ descents, and $\tau \mu = \sigma$. Then
\begin{equation*}
\sum_{i,j \geq 0} \frac{c_{ij}^d s^{i+1}
 t^{j+1}}{(1-s)^{n+1} (1-t)^{n+1}} = \sum_{a,b \geq 0} {n+ab-d-1
 \choose n} s^a t^b.
\end{equation*}
\end{prop}
\begin{proof} Since an $a$-shuffle followed by a $b$-shuffle is an $ab$-shuffle,
the formula \eqref{12} implies that \[ \sum_{\mu \in S_n}{n+a-d(\mu)-1
\choose n} \mu^{-1} \cdot \sum_{\tau \in S_n}{n+b-d(\tau)-1 \choose n}
\tau^{-1} = \sum_{\sigma \in S_n}{n+ab-d(\sigma)-1 \choose n}
\sigma^{-1}.\] Multiplying both sides by $s^a t^b$, summing over all $a,b
\geq 0$, and then taking the coefficient of $\sigma^{-1}$ on both sides
yields that
\begin{eqnarray*}
& & \sum_{a,b \geq 0} {n+ab-d-1
 \choose n} s^a t^b\\ & = & \sum_{(\tau,\mu) \atop \tau \mu=\sigma} \left[ \sum_{a
 \geq 0} s^a {n+a-d(\mu)-1 \choose n} \cdot \sum_{b
 \geq 0} t^b {n+b-d(\tau)-1 \choose n} \right] \\
 & = & \sum_{(\tau,\mu) \atop \tau \mu=\sigma}
 \frac{s^{d(\mu)+1}}{(1-s)^{n+1}} \frac{t^{d(\tau)+1}}{(1-t)^{n+1}}\\
 & = & \sum_{i,j \geq 0} \frac{c_{ij}^d s^{i+1}
 t^{j+1}}{(1-s)^{n+1} (1-t)^{n+1}}.
 \end{eqnarray*} \end{proof}

Recall that if a Markov chain has transition probabilities $P(i,j)$, its
formal time reversal with respect to a stationary measure $\pi$ is defined
to have transition probabilities $P^*(i,j) = \frac{P(j,i)
  \pi(j)}{\pi(i)}$.  This $P^*$ is a Markov transition matrix which
also has $\pi$ as stationary measure. A Markov chain $P$ is reversible
with respect to $\pi$ if and only if $P=P^*$.

Theorem \ref{identify} identifies the carries Markov chain with the formal
time reversal of a chain arising in the theory of riffle shuffles. As in
the introduction, $\pi$ denotes the distribution on $\{0,1,\dots,n-1\}$
defined by $\pi(j) = \frac{A(n,j)}{n!}$, where $A(n,j)$ is the number of
permutations in $S_n$ with $j$ descents.
\begin{theorem} \label{identify} Let a Markov chain on the symmetric
  group $S_n$ begin at the identity and proceed by successive independent
  $b$-shuffles. Then the number of descents of $\tau^{-1}$ forms a
  Markov chain with stationary distribution $\pi(j)=\frac{A(n,j)}{n!}$,
  and its formal time reversal with respect to $\pi$ is identical with
  the carries Markov chain.
\end{theorem}
\begin{proof} Let $d(\tau_r^{-1})$ denote the number of descents of
  the inverse of the permutation $\tau_r$ obtained after $r$
  independent $b$-shuffles. Corollary 2 of \cite{pd92} showed that
  $d(\tau_r^{-1})$ forms a Markov chain. Note that the stationary
distribution of this chain is given by $\pi(j)=\frac{A(n,j)}{n!}$, since
$\tau_r^{-1}$ tends to a uniform element of $S_n$ as $r \rightarrow
\infty$.

We compute the transition probabilities of the Markov chain formed by
$d(\tau_r^{-1})$. By \eqref{12}, $\mathbb{P}(d(\tau_{r-1}^{-1})=i) =
\frac{A(n,i)
    {n+b^{r-1}-i-1 \choose n}}{b^{(r-1)n}}$. Clearly
\begin{align*}
&\mathbb{P}\left(d(\tau_{r-1}^{-1})=i,d(\tau_r^{-1})=j\right)\\
& \quad =  \sum_{\sigma:d(\sigma^{-1})=i} \frac{{n+b^{r-1}-i-1 \choose n}}
{b^{(r-1)n}} \sum_{k
\geq 0} \sum_{\mu: d(\mu^{-1})=k \atop d(\sigma^{-1} \mu^{-1})=j}
\frac{{n+b-k-1 \choose n}}{b^n}.
\end{align*}
Thus
\begin{align*}
\mathbb{P}\left(d(\tau_r^{-1})=j|d(\tau_{r-1}^{-1})=i\right)
& = \frac{\mathbb{P}\left(d(\tau_{r-1}^{-1})=i,d(\tau_r^{-1})=j\right)}
{\mathbb{P}\left(d(\tau_{r-1}^{-1})=i\right)}\\
& = \frac{1}{A(n,i)} \sum_{\sigma: d(\sigma^{-1})=i} \sum_{k \geq 0}
\sum_{\mu: d(\mu^{-1})=k \atop d(\sigma^{-1} \mu^{-1})=j} \frac{{n+b-k-1
\choose n}}{b^n}.
\end{align*} In the notation of Proposition \ref{gen}, this is
\begin{equation*}
\frac{A(n,j)}{A(n,i)} \frac{1}{b^n} \sum_{k \geq 0} c_{ik}^j {n+b-k-1
\choose n}.
\end{equation*}
Letting $[x^h] f(x)$ denote the coefficient of $x^h$ in a series $f(x)$,
this can be rewritten as
\begin{align*}
\lefteqn{ [t^b] \frac{A(n,j)}{A(n,i)} \frac{1}{b^n} \sum_{k \geq 0}
 c_{ik}^j \frac{t^{k+1}}{(1-t)^{n+1}}} \\
&\quad =  [t^b s^{i+1}] \frac{A(n,j)}{A(n,i)} \frac{(1-s)^{n+1}}{b^n}
\sum_{i,k \geq 0} c_{ik}^j \frac{s^{i+1} t^{k+1}}{(1-s)^{n+1}
(1-t)^{n+1}}.
\end{align*}
By Proposition \ref{gen}, this is equal to
\begin{align*}
\lefteqn{ [t^b s^{i+1}] \frac{A(n,j)}{A(n,i)} \frac{(1-s)^{n+1}}{b^n}
 \sum_{a,d \geq 0} {n+ad-j-1 \choose n} s^a t^d} \\
& \quad =  [s^{i+1}] \frac{A(n,j)}{A(n,i)} \frac{(1-s)^{n+1}}{b^n} \sum_{a
\geq
0} {n+ab-j-1 \choose n} s^a \\
& \quad =  \frac{A(n,j)}{A(n,i)} \frac{1}{b^n} \sum_{l \geq 0} (-1)^l {n+1
\choose l} {n-1-j+(i+1-l)b \choose n}.
\end{align*}
This is equal to $\frac{\pi(j) P(j,i)}{\pi(i)}$ where $P$ is the
transition probability of the carries chain \eqref{11}.
\end{proof}

The next result gives a second, more direct, interpretation of the
transition probabilities of the carries chain.
\begin{theorem}\label{direct}
  The chance that the base-$b$ carries chain goes from $0$ to $j$ in $r$ steps is
  equal to the chance that a permutation in $S_n$ obtained by
  performing $r$ successive $b$-shuffles (started at the identity) has
  $j$ descents.
\end{theorem}
\begin{proof} By \eqref{12} and the fact that an $a$-shuffle followed by a $b$-shuffle
is an $ab$-shuffle, the chance that $r$ successive $b$-shuffles (started
at the identity) lead to a permutation with $j$ descents is
\begin{equation} \label{ed}
\sum_{i \geq 0} \frac{1}{b^{rn}} {n+b^r-i-1 \choose n} c_{ij}^0,
\end{equation}
where as in Proposition \ref{gen}, $c_{ij}^0$ denotes the number of
$\sigma \in S_n$ such that $d(\sigma^{-1})=i$ and $d(\sigma)=j$.

Proposition \ref{gen} gives that
\begin{equation*}
\sum_{i,k \geq 0} \frac{c_{ik}^0 s^{i+1}
t^{k+1}}{(1-s)^{n+1} (1-t)^{n+1}} = \sum_{a,d \geq 0} {n+ad-1 \choose n}
s^a t^d .
\end{equation*}
Taking the coefficient of $s^{b^r}$ on both sides gives that
\begin{equation*}
\sum_{i,k \geq 0} \frac{c_{ik}^0 {n+b^r-i-1 \choose n}
t^{k+1}}{(1-t)^{n+1}} = \sum_{d \geq 0} {n+b^r d -1 \choose n} t^d.
\end{equation*}
Comparing with equation \eqref{ed} gives that the chance that a
permutation obtained after $r$ successive $b$-shuffles has $j$ descents is
\begin{align*}
& \frac{1}{b^{rn}} [t^{j+1}] (1-t)^{n+1} \sum_{d \geq
0} {n+b^rd -1 \choose n} t^d\\
& \quad = \frac{1}{b^{rn}} \sum_{l \geq 0}
(-1)^l {n+1 \choose l}{n-1+(j+1-l)b^r \choose n}.
\end{align*} From \eqref{11}, this is equal to the carries
transition probability $P_{b^r}(0,j)$. By equation \eqref{hol}, this is
$P_b^r(0,j)$, as claimed.
\end{proof}

We conclude this section with two alternative derivations of the
stationary distribution of the carries chain. The following lemma will be
helpful. Stanley \cite{sta77} and Pitman \cite{Pi} give bijective proofs.

\begin{lemma} \label{hyper} Let $X_1,\dots,X_n$ be independent uniform
  $[0,1]$ random variables. Then for all integers $j$, $\mathbb{P}(j-1
  \leq \sum_{i=1}^n X_i < j)$ is equal to the probability that a
  uniformly chosen random permutation on $n$ symbols has $j$ descents.
\end{lemma}

As usual we let $P^r(0,j)$ denote the distribution on $\{0,1,\dots,n-1\}$
after $r$ steps of the carries chain (for the base $b$ addition of $n$
numbers) started from $0$.

\begin{theorem} \label{stat} (\cite{holte}) The stationary distribution
  $\pi$ of the carries chain satisfies $\pi(j)=\frac{A(n,j)}{n!}$,
  where $A(n,j)$ is the number of permutations in $S_n$ with $j$
  descents.
\end{theorem}

\begin{proof} By Holte \cite{holte}, $r$ steps of the base $b$ carries
  chain is equivalent to one step of the base $b^r$ carries chain.
  Letting $Y_1,\dots,Y_n$ be discrete i.i.d. uniforms on
  $\{0,1,\dots,b^r-1\}$, it follows that
\begin{equation*}
P^r(0,j) = \mathbb{P}\left(jb^r \leq \sum_{i=1}^n Y_i < (j+1)b^r\right).
\end{equation*}
Letting $U_1,\dots,U_n$ be continuous i.i.d. uniforms on $[0,b^r]$,
this implies that
\begin{align*}
P^r(0,j) & = \mathbb{P}\left(jb^r \leq \sum_{i=1}^n \lfloor
U_i\rfloor < (j+1)b^r\right) \\
& = \mathbb{P}\left(jb^r \leq \sum_{i=1}^n U_i - \sum_{i=1}^n (U_i - \lfloor U_i
\rfloor) < (j+1)b^r\right) \\
& = \mathbb{P}\left(j \leq \sum_{i=1}^n X_i - E < j+1\right).
\end{align*}
Here the $X_i=\frac{U_i}{b^r}$ are i.i.d. uniforms on $[0,1]$ and
$E=\frac{1}{b^r} \sum_{i=1}^n (U_i - \lfloor U_i \rfloor)$.

Although $E$ is not independent of the $X_i$'s, note that when $n$ is
fixed and $r\to\infty$, $E$ converges in probability to $0$.  Indeed,
this follows since $|E|\leq\frac{n}{b^r}$ with probability $1$.  Thus
Slutsky's theorem implies that
\begin{equation*}
\lim_{r\to\infty} P^r(0,j)=\mathbb{P}\left(j\leq\sum_{i=1}^n
X_i<j+1\right),
\end{equation*}
and the result follows from Lemma \ref{hyper}.
\end{proof}

A simple analytic way to find the stationary distribution uses the closed
form for $P^r(0,j)$. As $r$ tends to infinity,
\begin{equation*}
\frac{1}{b^{rn}}\binom{n-1+(j+1-l)b^r}{n} \rightarrow
\frac{(j+1-l)^n}{n!}.
\end{equation*}
Thus by \eqref{11} and \eqref{hol}, \begin{eqnarray*} P^r(0,j) & = &
\frac{1}{b^{rn}}\sum_{l=0}^{j} (-1)^l \binom{n+1}{l}
\binom{n-1+(j+1-l)b^r}{n} \\
& \rightarrow &
\frac{1}{n!}\sum_{l=0}^j(-1)^l\binom{n+1}{l}(j+1-l)^n=\frac{A(n,j)}{n!}.
\end{eqnarray*} The last equality is an identity, due to Euler, for the
$A(n,j)$ \cite{co}.

\section{Rates of Convergence}\label{sec3}

This section presents both upper and lower bounds on convergence to
stationarity for the equivalent Markov chains of \ref{sec2}. Theorem
\ref{up} shows that the descent set of a permutation (not just the number
of descents) is close to its stationary distribution after $r$
$b$-shuffles if $r=\log_b(n)+c$. This uses symmetric function theory.
Theorem \ref{newthm3.2} uses stochastic monotonicity to bound convergence
of the carries chain: it shows that at least $r=\frac{1}{2} \log_b(n)+c$
steps are needed and that $r=\log_b(n)+c$ steps suffice. Theorem
\ref{sharp} shows that for large $n$, $\frac{1}{2} \log_b(n)+c$ steps are
sufficient.

All of our results involve the total variation distance between
probability measures $P$ and $Q$ on a finite set $\mathcal{X}$, defined as
\begin{equation*}
\|P-Q\|_{\tv}=\dfrac12 \sum_x
|P(x)-Q(x)|=\max_{A\subseteq\mathcal{X}}|P(A)-Q(A)|.
\end{equation*}
\begin{theorem}\label{thm3.1}
  Consider the carries chain for base $b$ addition of $n$ numbers. Let
  $r= \lceil log_b(cn) \rceil$ with $c>0$. Let $P^r_0$ denote the
  distribution on $\{0,1,\dots,n-1\}$ given by taking $r$ steps in the
  carries chain, started from 0. Let $\pi$ be the stationary
  distribution of the carries chain. Then
\begin{equation*}
||P^r_0-\pi||_{\tv} \leq \dfrac12 \sqrt{e^{1/(2c^2)}-1}.
\end{equation*}
\end{theorem}
In fact, we prove a stronger result. This uses the notion of the descent
set of a permutation $\sigma$, defined as the set of $i$, $1 \leq i \leq
n-1$, such that $\sigma(i)>\sigma(i+1)$. For instance
$\bm{5}\,1\,\bm{3}\,2\,4$ has descent set $\{1,3\}$. Let
$\widetilde{P^r}(S)$ denote the probability that a permutation obtained
after the iteration of $r$ $b$-shuffles (or equivalently a single
$b^r$-shuffle) has descent set $S$, and let $\widetilde{\pi}(S)$ denote
the chance that a uniformly chosen random permutation has descent set $S$.
Theorem \ref{up} uses symmetric function theory to upper bound the total
variation distance between $\widetilde{P^r}$ and $\widetilde{\pi}$.
Chapter 7 of the text \cite{sta99} provides background on the concepts
used in the proof of Theorem \ref{up} (i.e. Young tableaux, the RSK
correspondence, and symmetric functions). In \cite{AD} it is shown that
the descent set is a Markov chain when cards are repeatedly $b$-shuffled.
\begin{theorem} \label{up}
Let $r= \lceil \log_b(cn) \rceil$ with $c>0$. Then with the notation of
Theorem \ref{thm3.1},
\begin{equation*}
||\widetilde{P^r}-\widetilde{\pi}||_{TV} \leq \dfrac12
\sqrt{e^{1/(2c^2)}-1}.
\end{equation*}
\end{theorem}
\begin{proof} We use the RSK correspondence which associates to a
  permutation $\sigma$ a pair of standard Young tableaux $(P,Q)$ called the
  insertion and recording tableau of $\sigma$ respectively. One says that a
  standard Young tableau $T$ has a descent at $i$ ($1 \leq i \leq
  n-1$) if $i+1$ is in a row lower than $i$ in $T$. We let $d(T)$
  denote the number of descents of $T$. By Lemma 7.23.1 of
  \cite{sta99}, the descent set of $\sigma$ is equal to the descent set of
  $Q(\sigma)$. This implies that
\begin{equation*}
\widetilde{\pi}(S) = \sum_{|\lambda|=n} \frac{f_{\lambda}(S)
 f_{\lambda}}{n!},
\end{equation*} where $f_{\lambda}$ is the number of standard Young
tableaux of shape $\lambda$, and $f_{\lambda}(S)$ is the number of
standard Young tableaux of shape $\lambda$ with descent set $S$.

From Theorem 3 of \cite{ful02}, the chance that $Q(\sigma)=T$ (for
$\sigma$ obtained from a $b^r$ shuffle) is
$s_{\lambda}(\frac{1}{b^r},\dots,\frac{1}{b^r})$ for any standard Young
tableau $T$. Here $b^r$ coordinates of the Schur function $s_{\lambda}$
are equal to $\frac{1}{b^r}$ and the rest are $0$. Thus,
\begin{align*}
||\widetilde{P^r}-\widetilde{\pi}||_{TV} &= \dfrac12 \sum_{S \subseteq
 \{1,\dots,n-1\}} \left| \widetilde{P^r}(S)-\widetilde{\pi}(S) \right|\\
&= \dfrac12 \sum_S\left| \sum_{|\lambda|=n} \left[ f_{\lambda}(S)
s_{\lambda}\left(\frac{1}{b^r},\dots,\frac{1}{b^r}\right) -
\frac{f_{\lambda}(S) f_{\lambda}}{n!} \right] \right| \\
&\leq \dfrac12 \sum_S \sum_{|\lambda|=n} \left| f_{\lambda}(S)
 s_{\lambda}\left(\frac{1}{b^r},\dots,\frac{1}{b^r}\right) - \frac{f_{\lambda}(S) f_{\lambda}}{n!} \right| \\
&= \dfrac12 \sum_{|\lambda|=n} \left|
s_{\lambda}\left(\frac{1}{b^r},\dots, \frac{1}{b^r}\right) - \frac{f_{\lambda}}{n!} \right| \sum_{S} f_{\lambda}(S) \\
&= \dfrac12 \sum_{|\lambda|=n} \left| f_{\lambda}
 s_{\lambda}\left(\frac{1}{b^r},\dots,\frac{1}{b^r}\right) - \frac{f_{\lambda}^2}{n!} \right|.
\end{align*}

By the Cauchy--Schwarz inequality, this is at most
\begin{align*}
& \dfrac12 \sqrt{ \sum_{|\lambda|=n}
\left[s_{\lambda}\left(\frac{1}{b^r},\dots,\frac{1}{b^r}\right)-\frac{f_{\lambda}}{n!} \right]^2
\sum_{|\lambda|=n} f_{\lambda}^2 } \\
& \quad = \dfrac12 \sqrt{ n! \sum_{|\lambda|=n}
\left[s_{\lambda}\left(\frac{1}{b^r},\dots,\frac{1}{b^r}\right)-\frac{f_{\lambda}}{n!} \right]^2 }.
\end{align*}
The functions $f_{\lambda}
s_{\lambda}(\frac{1}{b^r},\dots,\frac{1}{b^r})$ and
$\frac{f_{\lambda}^2}{n!}$ both define probability measures on the set
of partitions of size $n$; the first is the distribution on RSK shapes
after a $b^r$ riffle shuffle \cite{sta01}, and the second is known as
Plancherel measure. Hence the previous expression simplifies to
\begin{equation*}
\dfrac12 \sqrt{n! \sum_{|\lambda|=n}
s_{\lambda}\left(\frac{1}{b^r},\dots,\frac{1}{b^r} \right)^2 -1 }.
\end{equation*}
Let $[u^n] f(u)$ denote the coefficient of $u^n$ in a series $f(u)$. By
the Cauchy identity for Schur functions \cite[p.~322]{sta99},
\begin{align*}
\sum_{|\lambda|=n}s_{\lambda}\left(\frac{1}{b^r},\dots,\frac{1}{b^r}\right)^2
& = [u^n]\sum_{|\lambda| \geq 0} s_{\lambda}\left(\frac{u}{b^r},\dots,\frac{u}{b^r}\right)
s_{\lambda}\left(\frac{1}{b^r},\dots,\frac{1}{b^r}\right) \\
& = [u^n] \left( 1-\frac{u}{b^{2r}} \right)^{-b^{2r}} \\
& = b^{-2rn} {b^{2r}+n-1 \choose n}.
\end{align*}
Thus
\begin{equation*}
n!\sum_{|\lambda|=n} s_{\lambda}\left(\frac{1}{b^r},\dots,\frac{1}{b^r}\right)^2 - 1 =
\prod_{i=1}^{n-1} \left( 1+\frac{i}{b^{2r}} \right) -1 .
\end{equation*}
Since $\log(1+x) \leq x$ for $x>0$, it follows that
\begin{equation*}
\log \left(\prod_{i=1}^{n-1} \left( 1+\frac{i}{b^{2r}} \right) \right) =
\sum_{i=1}^{n-1} \log \left( 1+\frac{i}{b^{2r}} \right)
\leq \frac{{n \choose 2}}{b^{2r}}.
\end{equation*}
Thus
\begin{equation*}
\prod_{i=1}^{n-1} \left(1+\frac{i}{b^{2r}} \right) -1 \leq \exp \left(\frac{{n \choose
2}}{b^{2r}} \right) - 1.
\end{equation*}

Summarizing, it has been shown that
\begin{equation*}
||\widetilde{P^r}-\widetilde{\pi}||_{TV} \leq \dfrac12 \sqrt{ \exp
\left(\frac{{n \choose 2}}{b^{2r}} \right) - 1}.
\end{equation*}
If $b^r=cn$ with $c>0$, then $\frac{{n \choose 2}}{b^{2r}} \leq
\frac{1}{2c^2}$, which proves the result.
\end{proof}
\begin{proof}[Proof of Theorem \ref{thm3.1}]
Theorem \ref{direct} showed that the base-$b$ carries chain started from
$0$ is the same as the chain for the number of descents after successive
$b$-shuffles started from the identity. Thus Theorem \ref{up} also upper
bounds the total variation distance between $r$ iterations of the base-$b$
carries chain (started from $0$) and its stationary distribution.
\end{proof}

Next we give a different approach to proving convergence using stochastic
monotonicity and also give a lower bound. The arguments show that
$\log_bn+c$ steps suffice for convergence and that $\tfrac12\log_bn$ steps
are not enough.
\begin{theorem}\label{newthm3.2}
For $n \geq 3$, any starting state $i$, and any $r \geq 0$, the Markov
chain $P$ of \eqref{11} satisfies
\begin{equation*}
\|P^r(i,\cdot)-\pi\|_{\tv}\leq\left(\frac{n-1}{2}+i\right)\bigg/b^r.
\end{equation*}
Conversely, for any $\epsilon$, $0<\epsilon<1$, if $1 \leq r \leq \log_b
\left[ \frac{\epsilon \left| i-\frac{n-1}{2} \right|}{\sqrt{n}} \right]$,
then
\begin{equation*}
\|P^r(i,\cdot)-\pi\|_{\tv}\geq1-\epsilon.
\end{equation*}
\end{theorem}
\begin{proof} Recall that a Markov chain on $\{0,1,\dots,n-1\}$ is
  \textit{stochastically monotone} if for all $i\leq i'$,
  $P(i,\{0,\dots,j\})\geq P(i',\{0,\dots,j\})$ for all $j$. We
  show that $P$ is stochastically monotone by coupling. Consider two
  copies of the carries chain, one at $i$ and one at $i'$ with
  $i\leq i'$. Each chain proceeds by adding $n$ random base-$b$
  digits. Couple them by adding the \textit{same} digits to both. If
  the first process results in a carry of $k$, the second process
  results in a carry of $k$ or $k+1$. This implies stochastic
  monotonicity.

From Holte \cite[Th.~4]{holte} and the fact that $n \geq 3$, the right
eigenfunctions for eigenvalues $\frac{1}{b}$, $\frac{1}{b^2}$ can be taken
as
\begin{equation*}
f_1(i)=i-\frac{n-1}{2},\qquad f_2(i)=i^2-(n-1)i+\frac{(n-2)(3n-1)}{12}.
\end{equation*}

The upper bound follows from stochastic monotonicity and the first
eigenvector via \cite[Th.~2.1]{pd170}. For the lower bound, note that
$f_1^2=f_2+A$, with $A=\frac{n+1}{12}$. This, and a simple computation
show that
\begin{equation*}
\int\left(f_1(x)-f_1(y)\right)^2P(x,dy)=
\left(1-\frac{1}{b}\right)^2f_1^2(x)+ A\left(1-\frac{1}{b^2}\right).
\end{equation*}
This is the required input for the lower bound, using
\cite[Th.~2.3]{pd170}. One obtains that
$\|P^r(i,\cdot)-\pi\|_{\tv}\geq1-\epsilon$ for $r \leq \log_b \left[
\frac{\epsilon \left| i-\frac{n-1}{2} \right|}{\sqrt{8(n+1)/12}} \right]$,
and the result follows since $\frac{8(n+1)}{12} \leq n$ when $n \geq 3$.
\end{proof}
\begin{rem}
  The argument for stochastic monotonicity does not depend on the
  assumption that the digits are uniform and independently
  distributed. Any joint distribution within a column (with columns
  independent) leads to a stochastically monotone Markov chain. In
  \cite{pd173} it is shown that the transition matrix $P$ is totally
  positive of order $2$. This implies stochastic monotonicity via
  \cite[Prop.~1.3.1, p.~22]{kar}.
\end{rem}

To close this section, we prove that $\frac{1}{2} \log_b(n)+c$ steps are
sufficient for total variation convergence when $n$ is large.

\begin{theorem} \label{sharp} With the notation of Theorem \ref{thm3.1},
there is a constant $B>0$ (independent of $n,b,c>1$) such that for
$r=\frac{1}{2} \log_b(nc)$, \[ ||P^r_0-\pi||_{TV} \leq \frac{B}{\sqrt{n}}
+ \frac{B}{\sqrt{c}}.\] \end{theorem}

\begin{proof} From Theorem 4.3 of \cite{pd173}, there is $j^* \in
\{0,1,\cdots,n-1 \}$ such that $P^r(0,j) \geq \pi(j)$ for $0 \leq j \leq
j^*$, and that $P^r(0,j) \leq \pi(j)$ for $j^*+1 \leq j \leq n-1$. Thus
\begin{equation} \label{1} ||P^r_0-\pi||_{TV} = P^r_0 \{0,1,\cdots,j^*\} -
\pi \{0,1,\cdots,j^*\}.\end{equation} From the proof of Theorem
\ref{stat}, $P^r(0,j) = \mathbb{P} \left(jb^r \leq \sum_{i=1}^n Y_i <
(j+1)b^r \right)$ with $Y_i$ i.i.d. uniform on $\{0,1,\cdots,b^r-1\}$.
From Lemma \ref{hyper}, $\pi(j) = \mathbb{P} \left( j \leq \sum_{i=1}^n
U_i < j+1 \right)$ with $U_i$ i.i.d. uniform on $[0,1]$. Thus
\[ P^r_0(j) = \mathbb{P} \left( \lfloor \frac{1}{b^r} \sum_{i=1}^n Y_i \rfloor=j
\right) {\rm and} \ P^r_0 \{0,1,\cdots,j^*\} = \mathbb{P} \left(
\frac{1}{b^r} \sum_{i=1}^n Y_i < j^*+1 \right),\]
\[ \pi(j) = \mathbb{P} \left( \lfloor \sum_{i=1}^n U_i \rfloor =j \right) {\rm
and} \ \pi \{0,1,\cdots,j^*\} = \mathbb{P} \left( \sum_{i=1}^n U_i < j^*+1
\right).\]

From the above considerations, we have
\begin{equation} \label{2} ||P^r_0 - \pi||_{TV} \leq \sup_x \left| \mathbb{P}
\left( \frac{1}{b^r} \sum_{i=1}^n Y_i < x \right) - \mathbb{P} \left(
\sum_{i=1}^n U_i <x \right) \right|. \end{equation} Let
$\mu_n=\frac{n}{2}, \sigma_n^2=\frac{n}{12}$ and $\nu_n=\frac{n}{2}
\frac{b^r-1}{b^r}, \tau_n^2=\frac{n}{12} \frac{b^{2r}-1}{b^{2r}}$. The
right-hand side of \eqref{2} is
\begin{eqnarray*}
& & \sup_x \left| \mathbb{P} \left[ \frac{\left( \frac{1}{b^r}
\sum_{i=1}^n Y_i - \mu_n \right)}{\sigma_n} < \frac{(x-\mu_n)}{\sigma_n}
\right] - \mathbb{P} \left[ \frac{\left( \sum_{i=1}^n U_i - \mu_n \right)}
{\sigma_n} <
\frac{(x-\mu_n)}{\sigma_n} \right] \right|\\
& \leq & \sup_y \left| \mathbb{P} \left[ \frac{\left( \frac{1}{b^r}
\sum_{i=1}^n Y_i - \mu_n \right)}{\sigma_n} < y \right] - \Phi(y) \right|
+ \sup_y \left| \mathbb{P} \left[ \frac{\left( \sum_{i=1}^n U_i - \mu_n
\right)}{\sigma_n}
< y \right] - \Phi(y) \right| \\
& = & I + II. \end{eqnarray*} Here $\Phi(y) = \frac{1}{\sqrt{2 \pi}}
\int_{-\infty}^y e^{-t^2/2} dt$ denotes the cumulative distribution
function of the normal distribution.

From the usual Berry-Esseen bound, $II \leq B_1/\sqrt{n}$ with $B_1$
involving the second and third moments of the uniform on $[0,1]$,
uniformly bounded. Rewrite $I$ as \begin{eqnarray*} & & \sup_y \left|
\mathbb{P} \left[ \frac{ \left( \frac{1}{b^r} \sum_{i=1}^n Y_i - \nu_n
\right)}{\tau_n} < \frac{(\sigma_n
y + \mu_n - \nu_n)}{\tau_n} \right] - \Phi(y) \right| \\
& \leq & \sup_z \left| \mathbb{P} \left[ \frac{\left( \frac{1}{b^r} \sum_i
Y_i - \nu_n \right)}{\tau_n} \leq z \right] - \Phi(z) \right| + \sup_z
\left| \Phi(z) - \Phi(a_1 z + a_2) \right| \end{eqnarray*} with $a_1
=\tau_n/\sigma_n$, $a_2 = (\nu_n - \mu_n) / \sigma_n$. Using the
Berry-Esseen bound again, the first term is bounded above by
$B_2/\sqrt{n}$ with $B_2$ involving the ratio \[ \mathbb{E} \left|
\frac{Y_1}{b^r} - \frac{b^r-1}{2b^r} \right|^3 \left/ \left( \mathbb{E}
\left| \frac{Y_1}{b^r} - \frac{b^r-1}{2b^r} \right|^2 \right)^{3/2}.
\right.\] This is uniformly bounded in $b,n,c$. To bound the final term,
we use the following inequality: for any $\mu \in \mathbb{R}, \sigma^2 \in
\mathbb{R}^+$, \begin{equation} \label{3} \sup_z \left| \Phi(z) -
\Phi(\sigma z + \mu) \right| \leq |\sigma^2-1| + |\mu| \frac{\sqrt{2
\pi}}{4}. \end{equation}

An elegant proof of \eqref{3} using Stein's identity was communicated by
Sourav Chatterjee. Let $W$ be Normal$(\mu,\sigma^2)$ and $Z$ be
Normal$(0,1)$. For any bounded $f$ with a bounded, piecewise continuous
derivative, $\mathbb{E}(Wf(W)) = \mu \mathbb{E}(f(W)) + \sigma^2
\mathbb{E}(f'(W))$ (Stein's identity being used). Thus \[
\mathbb{E}(Wf(W)-f'(W)) = \mu \mathbb{E}(f(W)) + (\sigma^2-1)
\mathbb{E}(f'(W)).\] As in \cite[p.~22]{Stn}, choose $f_{w_0}$ so that for
all $w$, one has \begin{equation} \label{4} wf_{w_0}(w) - f_{w_0}'(w) =
\delta_{w \leq w_0} - \Phi(w_0). \end{equation} Here $w_0$ is fixed. Stein
shows that $|f_{w_0}(w)| \leq \frac{\sqrt{2 \pi}}{4}$ for all $w$, and
that $|f_{w_0}'(w)| \leq 1$ for all $w$. Taking expectations in \eqref{4}
proves \eqref{3}. Taking $\sigma^2 = \left( 1-\frac{1}{b^{2r}} \right)$,
$\mu=- \frac{\sqrt{3n}}{b^r} = -\sqrt{ \frac{3}{c}}$, it follows that \[
\sup_z \left| \Phi(z) - \Phi(a_1 z + a_2) \right| \leq B_3 / \sqrt{c} \]
with $B_3$ independent of $n,b,c>1$.
\end{proof}
\begin{rem}
  If $W$ is Normal$(\nu,\tau^2)$ and $Z$ is Normal$(\mu,\sigma^2)$, the
  bound \eqref{3} shows that the Kolmogorov distance between their
  distributions is at most \[ \min \left( \frac{|\mu-\nu|}{\tau}
  \frac{\sqrt{2 \pi}}{4} + \left|\frac{\sigma^2}{\tau^2}-1 \right|,
\frac{|\mu-\nu|}{\sigma}
  \frac{\sqrt{2 \pi}}{4} + \left|\frac{\tau^2}{\sigma^2}-1 \right|
  \right).\]
\end{rem}

\section{Signed Permutations}\label{sec4}

Let $B_n$, the hyperoctahedral group, be represented as signed
permutations. Thus $B_n$ has $2^n n!$ elements. We associate elements of
$B_n$ to arrangements of a deck of $n$ cards with cards allowed to be face
up or face down. A natural analog of the Gilbert--Shannon--Reeds shuffling
model was studied in \cite{pd92}; the deck is cut approximately in half,
the top half turned face up, and the two halves are riffled together
according to the G--S--R prescription.  These shuffles have similarly neat
combinatorial properties which allow sharp analysis of mixing times. Of
course, shuffling is a natural algebraic operation and type $B$ shuffles
have been studied from an algebraic viewpoint (with applications to
Hochschild homology) in \cite{ber}, \cite{nantel}, and \cite{ful01}. This
section develops a corresponding carries process in rough parallel with
\ref{sec2}. We also give an application to the theory of rounding.

From the previous sections, we see that the key idea is to use the fact
that an $a$-shuffle followed by a $b$-shuffle is equivalent to an
$ab$-shuffle. A hyperoctahedral analog of $(2a+1)$-shuffles was considered
in \cite{ber} (see also \cite{ful01} for connections with the affine Weyl
group). A $(2a+1)$-shuffle is defined by multinomially cutting the deck
into $2a+1$ piles, then flipping over the even numbered piles, and
riffling them together.

View $B_n$ as the signed permutations on $n$ symbols, using the linear
ordering \[ 1<2<\dots<n<-n<\dots<-2<-1.\] Say that
\begin{enumerate}
\item $\sigma$ has a descent at position $i$ $(1 \leq i \leq n-1)$ if
$\sigma(i) > \sigma(i+1)$.
\item $\sigma$ has a descent at position $n$ if $\sigma(n)<0$.
\end{enumerate}
For example, $-1 \ -2 \ -3\in B_3$ has three descents. Let
$\overline{A(n,j)}$ denote the number of elements of $B_n$ with $j$
descents. The Bergerons \cite{ber} give analogs of basic properties of
riffle shuffles. More precisely, they show that if a Markov chain on the
hyperoctahedral group begins at the identity and proceeds by successive
independent $(2b+1)$-shuffles, then
\begin{itemize}
\item The chance of obtaining the signed permutation $\tau$ after $r$ steps
is \begin{equation} \label{41} \frac{{n+
\frac{(2b+1)^r-1}{2}-d(\tau^{-1})\choose n}}{(2b+1)^{rn}}.
\end{equation}
\item A $(2a+1)$-shuffle followed by a
$(2b+1)$-shuffle is equivalent to a $(2a+1)(2b+1)$-shuffle.
\end{itemize} Using these gives a type $B$ analog of Proposition \ref{gen}. Gessel also
has an unpublished proof of Proposition \ref{genB} using $P$-partitions.
\begin{prop} \label{genB} Let $\sigma \in B_n$ have $d$ descents. Let
  $c_{ij}^d$ be the number of ordered pairs $(\tau,\mu)$ of elements
  of $B_n$ such that $\tau$ has $i$ descents, $\mu$ has $j$ descents,
  and $\tau \mu = \sigma$. Then
\begin{equation*}
\sum_{i,j \geq 0} \frac{c_{ij}^d s^i t^j}{(1-s)^{n+1} (1-t)^{n+1}} =
\sum_{a,b \geq 0} {n+2ab+a+b-d \choose n} s^a t^b .
\end{equation*}
\end{prop}
\begin{proof} Since a $(2a+1)$-shuffle followed by a $(2b+1)$-shuffle is
equivalent to a $(2a+1)(2b+1)$-shuffle, the $r=1$ case of \eqref{41} gives
that \[ \sum_{\mu \in B_n}{n+a-d(\mu) \choose n} \mu^{-1} \cdot \sum_{\tau
\in B_n} {n+b-d(\tau)\choose n} \tau^{-1} =  \sum_{\sigma \in
B_n}{n+2ab+a+b-d(\sigma) \choose n} \sigma^{-1}.\] As in the proof of
Proposition \ref{gen}, one multiplies both sides by $s^a t^b$, sums over
all $a,b \geq 0$, and then takes the coefficient of $\sigma^{-1}$ on both
sides to obtain the result.
\end{proof}

Next we define a ``type $B$'' carries process, to which we will relate
the type $B$ hyperoctahedral shuffle. This is defined as the usual
carries process, where one adds $n$ length $m$ numbers base $2b+1$,
and to these adds the length $m$ number $(b,b,\dots,b)$. Note that the
state space of the type $B$ carries chain is $\{0,1,\dots,n\}$ (for
usual carries, the most one can carry is $n-1$). For example when
$b=1$ (so $2b+1=3$), adding $222$ and $201$ followed by appending
$111$ gives
\begin{equation*}
\begin{array}{llll}
{\it 2}& {\it 1} & {\it 1}& \\
 &2&2&2\\
 &2&0&1\\
 &1&1&1\\ \hline
2&0&1&1
\end{array}
\end{equation*}
with carries $\kappa_0=0, \kappa_1=1, \kappa_2=1, \kappa_3=2$.

\begin{theorem} \label{transB}
For $0\leq i$, $j\leq n$,
\begin{enumerate}
\item The transition probabilities of the type $B$ carries chain are
\begin{equation*}
P(i,j) = \frac{1}{(2b+1)^n} \sum_{l \geq 0} (-1)^l {n+1 \choose l}
{n+(j-l)(2b+1)+b-i \choose n} .
\end{equation*}

\item The $r$-step transition probabilities of the type $B$ carries
  chain are
\begin{equation*}
P^r(i,j) = \frac{1}{(2b+1)^{rn}} \sum_{l \geq 0} (-1)^l {n+1 \choose l}
{n+(j-l)(2b+1)^r+\frac{(2b+1)^r-1}{2}-i \choose n} .
\end{equation*} (i.e. one replaces $2b+1$ by $(2b+1)^r$ in part
1).
\end{enumerate}
\end{theorem}
\begin{proof} From the definition of the type $B$ carries chain,
\begin{equation*}
P(i,j)=\mathbb{P}\left( j(2b+1)-b\leq i+X_1+\dots+X_n\leq j(2b+1)+b\right)
\end{equation*}
where $X_1,\dots,X_n$ are independent and identically distributed discrete
uniform random variables in $\{0,1,\cdots,2b\}$.  Equivalently,
\begin{equation*}
P(i,j)=(2b+1) \cdot \mathbb{P}\left(i+X_1+\dots+X_n+Y=j(2b+1)+b\right),
\end{equation*}
where $X_1,\dots,X_n,Y$ are i.i.d. discrete uniforms in
$\{0,1,\cdots,2b\}$. Letting $[x^h] f(x)$ denote the coefficient of $x^h$
in a series $f(x)$, it follows that
\begin{align*}
P(i,j) & = \frac{1}{(2b+1)^n} \left[x^{j(2b+1)+b-i}\right]
\left( \frac{1-x^{2b+1}}{1-x} \right)^{n+1} \\
& = \frac{1}{(2b+1)^n} \sum_{l \geq 0} (-1)^l {n+1 \choose l}
\left[x^{(j-l)(2b+1)+b-i}\right] \left( \frac{1}{1-x} \right)^{n+1} \\
& = \frac{1}{(2b+1)^n} \sum_{l \geq 0} (-1)^l {n+1 \choose l}
{n+(j-l)(2b+1)+b-i \choose n}.
\end{align*}
Thus the first part is proved.

To prove the second half of the theorem, we show that $r$ steps of the
base-$(2b+1)$ carries chain is equivalent to one step of the base
$(2b+1)^r$ carries chain. To compute the carry after $r$ steps of the type
$B_n$ carries chain, add $b \ ( 1 + (2b+1) + \dots + (2b+1)^{r-1})$ to the
sum of $n$ length $r$ numbers base $2b+1$. To compute the carry after one
step of the type $B_n$ base $(2b+1)^r$ carries chain, add
$\frac{(2b+1)^r-1}{2}$ to the sum of $n$ length $1$ numbers base
$(2b+1)^r$. These computations are equivalent, so the result follows by
replacing $2b+1$ by $(2b+1)^r$ in part 1.
\end{proof}

Now we relate hyperoctahedral shuffles to type $B$ carries. In what
follows, $\pi$ denotes the distribution on $\{0,1,\dots,n\}$ defined by
$\pi(j)= \frac{\overline{A(n,j)}}{2^n n!}$.
\begin{theorem} \label{identifyB} Let a Markov chain on the
  hyperoctahedral group $B_n$ begin at the identity and proceed by
  successive independent $(2b+1)$-shuffles. Then the number of descents
  of $\tau^{-1}$ forms a Markov chain, and its formal time reversal
  with respect to its stationary distribution $\pi$ is identical with
  the carries Markov chain.
\end{theorem}

\begin{proof} Let $\tau_r$ be the element of $B_n$ obtained after $r$
  independent $b$-shuffles (started at the identity). Arguing as in
  the proof of Theorem \ref{identify} gives that $d(\tau_r^{-1})$ forms a
  Markov chain with transition probabilities
\begin{equation*}
\mathbb{P}\left( d(\tau_r^{-1})=j | d(\tau_{r-1}^{-1})=i\right) =
\frac{\overline{A(n,j)}}{\overline{A(n,i)} (2b+1)^n} \sum_{k \geq 0}
c_{ik}^j {n+b-k \choose n} .
\end{equation*}
Here $c_{ik}^j$ is as in the statement of Proposition \ref{genB}.

Letting $[x^h] f(x)$ denote the coefficient of $x^h$ in a series
$f(x)$, the transition probability in the previous paragraph can be
written as
\begin{align*}
& [t^b] \frac{\overline{A(n,j)}}{\overline{A(n,i)} (2b+1)^n} \sum_{k \geq
0}
c_{ik}^j \frac{t^k}{(1-t)^{n+1}} \\
& \qquad = \left[t^b s^i\right] \frac{\overline{A(n,j)}}{\overline{A(n,i)}
(2b+1)^n} (1-s)^{n+1} \sum_{i,k \geq 0} c_{ik}^j \frac{s^i
t^k}{(1-s)^{n+1} (1-t)^{n+1}}.
\end{align*}
By Proposition \ref{genB} this is equal to
\begin{align*}
& [t^b s^i] \frac{\overline{A(n,j)}}{\overline{A(n,i)} (2b+1)^n}
(1-s)^{n+1}
\sum_{a,c \geq 0} {n+2ac+a+c-j \choose n} s^a t^c \\
& \quad\qquad = \frac{\overline{A(n,j)}}{\overline{A(n,i)} (2b+1)^n}
\sum_{l \geq 0}
(-1)^l {n+1 \choose l} \left[t^b\right] \sum_{c \geq 0} {n+(i-l)(2c+1)+c-j \choose n}t^c \\
& \quad\qquad = \frac{\overline{A(n,j)}}{\overline{A(n,i)} (2b+1)^n}
\sum_{l \geq 0} (-1)^l {n+1 \choose l} {n+(i-l)(2b+1)+b-j \choose n}.
\end{align*}
Comparing with Theorem \ref{transB}, this is equal to $\frac{\pi(j)
  P(j,i)}{\pi(i)}$, as needed.
\end{proof}

The next theorem is easily proved by the technique used to prove Theorem
\ref{direct} (using Proposition \ref{genB} instead of Proposition
\ref{gen}).

\begin{theorem} \label{analogo} The chance that the type $B$ carries
chain goes from $0$ to $j$ in $r$ steps is equal to the chance that an
element of $B_n$ obtained by performing $r$ successive $(2b+1)$-shuffles
(started at the identity) has $j$ descents.
\end{theorem}

The following corollary is immediate from Theorem \ref{analogo}.
\begin{cor} \label{statB} The stationary distribution of the type $B$
  carries chain is given by $\pi(j)= \frac{\overline{A(n,j)}}{2^n n!}$,
  where $\overline{A(n,j)}$ is the number of signed permutations on $n$
  symbols with $j$ descents.
\end{cor}
Corollary \ref{exp} gives a closed formula for $\overline{A(n,j)}$. (This
can also be obtained by combining Proposition \ref{round} below with
equation (19) of \cite{cha}).

\begin{cor} \label{exp}
\[ \overline{A(n,j)}= \sum_{l=0}^j (-1)^l {n+1 \choose l}
\left( 2j-2l+1 \right)^n.\]
\end{cor}

\begin{proof} Let $r \rightarrow \infty$ in part 2 of Theorem
\ref{transB}, and apply Corollary \ref{statB}. \end{proof}

As an application of the above results, we give a new proof of the
following lovely fact from \cite{schm} (see also Section 9 of \cite{cha}
for closely related results). Note that it can be interpreted as computing
the chance that the sum of $n$ i.i.d. uniforms on $[0,1]$, when rounded to
the nearest integer, is equal to $j$.
\begin{prop} \label{round} Let $U_1,\dots,U_n$ be independent,
  identically distributed continuous uniform random variables in
  $[0,1]$. Then \[ \mathbb{P} \left(j-\frac{1}{2} \leq U_1+\cdots+U_n \leq
  j+\frac{1}{2} \right) = \frac{\overline{A(n,j)}}{2^n n!}.\]
\end{prop}
\begin{proof} Let $X_1,\dots,X_n$ be i.i.d. discrete uniforms on
  $\{0,1,\dots,2b\}$.  From the definition of the type $B_n$ base-$(2b+1)$
  carries chain,
\begin{equation*}
P(0,j)= \mathbb{P}\left( j(2b+1)-b\leq\sum_{i=1}^n X_i < j(2b+1)+b+1
\right) .
\end{equation*}
Let $Y_1,\dots,Y_n$ be i.i.d. continuous uniforms on $[0,2b+1]$. Then it
follows that
\begin{align*}
P(0,j)& = \mathbb{P}\left( j(2b+1)-b \leq\sum_{i=1}^n
\lfloor Y_i \rfloor < j(2b+1)+b+1 \right) \\
& = \mathbb{P} \left( (2b+1)\left(j-\dfrac12\right) \leq \sum_{i=1}^n Y_i -
\sum_{i=1}^n \left(Y_i -\lfloor Y_i \rfloor\right) - \dfrac12< (2b+1)\left(j+\dfrac12\right) \right) \\
& = \mathbb{P}\left( j -\dfrac12\leq\sum_{i=1}^n U_i - E < j +\dfrac12 \right).
\end{align*}
Here the $U_i = \frac{Y_i}{2b+1}$ are i.i.d. continuous uniforms on
$[0,1]$, and
\begin{equation*}
E = \frac{1}{2b+1}
\sum_{i=1}^n \left(Y_i - \lfloor Y_i \rfloor\right) + \frac{1}{2(2b+1)}.
\end{equation*}

Note that when $n$ is fixed and $b \to\infty$, $E$ converges to $0$
with probability $1$. Thus Slutsky's theorem implies that
\begin{equation*}
\lim_{b \to\infty} P(0,j) = \mathbb{P} \left( j - \dfrac12 \leq
\sum_{i=1}^n U_i < j + \dfrac12 \right).
\end{equation*} However by Theorem
\ref{transB} and Corollary \ref{statB},
\begin{equation*}
\lim_{b \to\infty} P(0,j) = \pi(j) = \frac{\overline{A(n,j)}}{2^n n!},
\end{equation*}
which completes the proof.
\end{proof}

\section{Two Final Topics}\label{sec5}

The carries matrix also comes up in studying sections of generating
functions via the Veronese map. The large $n$ limit of the carries
\textit{process} is well approximated by a classical auto-regressive
process.

\subsection{Eulerian polynomials and Hilbert series of Veronese
  subrings}\label{sub51}

Some natural sequences $a_k,\,0\leq k<\infty$ have generating functions
\begin{equation*}
\sum_{k=0}^\infty a_k x^k=\frac{h(x)}{(1-x)^{n+1}}
\end{equation*}
with $h(x)=h_0+h_1x+\dots+h_{n+1}x^{n+1}$ a polynomial of degree at most
$n+1$. Suppose we are interested in every $b$th term $\{a_{bk}\},\ 0\leq
k<\infty$. It is not hard to see that
\begin{equation*}
\sum_{k=0}^\infty a_{bk}x^k=\frac{h^{<b>}(x)}{(1-x)^{n+1}}
\end{equation*}
for another polynomial $h^{<b>}(x)$ of degree at most $n+1$. The study of
these generating functions arises naturally in algebraic geometry
\cite{eis} and lattice point enumeration \cite{beck}.

Brenti and Welker \cite{bren} show that the $i$th coefficient of
$h^{<b>}(x)$ satisfies
\begin{equation*}
h_i^{<b>}=\sum_{j=0}^{n+1}C(i,j)h_j
\end{equation*}
with $C$ an $(n+2)\times(n+2)$ matrix with $(i,j)$ entry ($0\leq i,j\leq
n+1$) equal to the number of solutions to $a_1+\dots+a_{n+1}=ib-j$ where
$0\leq a_l \leq b-1$ are integers. In \cite{pd173} we show that the
$n\times n$ matrix given by deleting the first and last rows and columns
of $C$, then multiplying by $b^{-n}$ and taking the transpose is precisely
the carries matrix $(P(i,j))$ of \eqref{11}.

Since iterates of the carries chain converge, the matrix $C(i,j)$ has nice
limiting behavior. Brenti and Welker \cite{bren} show that the zeros of
$h^{\langle b\rangle}$ converge and Beck-Stapledon \cite{beck} show that
the zeros converge to the zeros of the $n$th Eulerian polynomial $p_n(x)$,
defined as $\sum_{j \geq 0} A(n,j) x^{j+1}$, where $A(n,j)$ is the number
of permutations in $S_n$ with $j$ descents. The following is a refinement.
\begin{theorem} \label{lim} Suppose that $h(1) \neq 0$ and let $p_n(x)$
be the $n$th Eulerian polynomial. Then as $b \rightarrow \infty$ with $n$
fixed,
\begin{equation*}
\frac{h^{<b>}(x)}{b^n \cdot h(1)} \longrightarrow \frac{p_n(x)}{n!}
\end{equation*}
\end{theorem}
\begin{proof} Let $[y^k]f(y)$ denote the coefficient of $y^k$ in a power
series $f(y)$. Then the definition of $C(i,j)$ gives that
\begin{align*}
C(i,j) & = [y^{ib-j}] \frac{(1-y^b)^{n+1}}{(1-y)^{n+1}} \\
& = \sum_{l \geq 0} (-1)^l {n+1 \choose l} \left[y^{(i-l)b-j}\right]\frac{1}{(1-y)^{n+1}}\\
& = \sum_{l \geq 0} (-1)^l {n+1 \choose l} {n+(i-l)b-j \choose n}.
\end{align*}
Supposing that $1 \leq i \leq n$, it follows that
\begin{equation} \label{e9}
\lim_{b \to\infty} \frac{C(i,j)}{b^n} = \frac{1}{n!} \sum_{l \geq 0}
(-1)^l {n+1 \choose l} (i-l)^n = \frac{A(n,i-1)}{n!}
\end{equation}
where the second equality uses a well-known formula for Eulerian numbers
\cite{co}. Since $C(0,j) = \delta_{0,j}$ and $C(n+1,j)=\delta_{n+1,j}$,
clearly
\begin{equation} \label{e10}
\lim_{b\to\infty} \frac{C(0,j)}{b^n} =
\lim_{b\to\infty}\frac{C(n+1,j)}{b^n} = 0.
\end{equation}
Combining equations \eqref{e9} and \eqref{e10} yields that
\begin{align*}
\lim_{b\to\infty} \frac{h^{<b>}(x)}{b^n \cdot h(1)} & = \lim_{b\to\infty}
\frac{ \sum_{i=0}^{n+1} \left[ \sum_{j=0}^{n+1} C(i,j)
h_j \right] x^i}{b^n \cdot h(1)}\\
& = \frac{\sum_{i=1}^n \left[ A(n,i-1) \sum_{j=0}^{n+1} h_j \right]x^i}{n! \cdot h(1)} \\
& = \frac{p_n(x)}{n!}.
\end{align*}
\end{proof}

Here is an example. The coordinate ring $R$ of a projective variety in
$n+1$ variables decomposes into its graded pieces $R_k$, $0\leq
k\leq\infty$ and the Hilbert series has the form \cite[Th.~11.1]{ati}
\begin{equation*}
\sum_{k=0}^\infty\text{dim}(R_k) x^k=\frac{h(x)}{(1-x)^{n+1}}.
\end{equation*}
The \textit{$b$th Veronese embedding} replaces the variables by all degree
$b$ monomials in these variables. (If $b=3$, $\{x,y\}$ are replaced by
$x^3, x^2y,xy^2,y^3$.) The image of the coordinate ring has Hilbert series
$h^{\langle b\rangle}(x)/(1-x)^{n+1}$. As a simple special case, the full
projective space has coordinate ring $\mathbb{C}[x_1\dots x_{n+1}]$. The
degree $k$ homogeneous polynomials have dimension $\binom{n+k}{n}$ and
\begin{equation}
\sum_{k=0}^\infty\binom{n+k}{n}x^k=\frac{1}{(1-x)^{n+1}}. \label{51}
\end{equation}
When $n+1=2$,
\begin{equation*}
\sum_{k=0}^\infty(k+1)x^k=\frac{1}{(1-x)^2}\quad\text{and}\quad\sum_{k=0}^\infty(bk+1)x^k=\frac{(b-1)x+1}{(1-x)^2}.
\end{equation*}
When $n+1=3$,
\begin{gather*}
\sum_{k=0}^\infty\binom{k+2}{2}x^k=\frac{1}{(1-x)^3}\quad\text{and}\\
\sum_{k=0}^\infty\binom{bk+2}{2}x^k=\frac{x^2\left(b(b-3)+2\right)+x\left(b(b+3)-4\right)+2}{2(1-x)^3}.
\end{gather*} Dividing the right-hand sides of these expressions by $b^n \cdot h(1)$
(here $b$ and $b^2$ respectively), then multiplying by $(1-x)^{n+1}$ and
letting $b \rightarrow \infty$, gives $p_n(x)/n!$ (here $x$ and
$(x^2+x)/2$ respectively).

\subsection{Autoregressive approximation}\label{sub5.2}

This section studies the large $n$ limit of the carries process and shows
it is well approximated by a classical autoregressive process. Throughout,
we work with a general base $b$ and let $n$ be the number of summands. Let
$\kappa_0=0,\kappa_1,\kappa_2,\dots$ be the carries process on
$\{0,1,\dots,n-1\}$. Let $Y_t=(\kappa_t-n/2)/\sqrt{n/12}$,
$t=0,1,2,\dots$. Theorem \ref{thm5.2} relates $Y_t$ to a Gaussian
autoregressive process $W_0,W_1,W_2,\dots$ defined by $W_0=-\sqrt{3n}$,
$W_{t+1}=\frac{W_t}{b} +\epsilon_t$, with the $\epsilon_t$ independent
Normal$(0,1-\frac{1}{b^2})$ random variables.

\begin{theorem}\label{thm5.2}
Let $P_n$ be the law of the process $Y_t$, $0\leq t<\infty$, on
$\real^\infty$. Let $Q$ be the law of the process $W_0,W_1,\dots$ on
$\real^\infty$. Then $P_n\Rightarrow Q$ as $n \rightarrow \infty$.
\end{theorem}

The following lemma will be helpful for proving Theorem \ref{thm5.2}.

\begin{lemma} \label{rel} The base $b$ carries process can be represented as
follows: \begin{equation} If \ \kappa_t = r \mod b, \ \ \ let \
\kappa_{t+1}=\frac{\kappa_t - r}{b} + \epsilon_{t+1} \end{equation} where
$\mathbb{P}(\epsilon_{t+1}=k)$ is the chance that the sum of $n+1$
independent discrete uniforms on $\{0,1,\cdots,b-1\}$ is equal to
$bk+b-r-1$, given that the sum is congruent to $b-r-1$ mod $b$.
\end{lemma}

\begin{proof} From page 140 of Holte \cite{holte} one can write the
carries transition probability as
\begin{equation}
P(i,j)=\frac{1}{b^n} [x^{(j+1)b-i-1}] (1+x+\cdots+x^{b-1})^{n+1}
\label{54}
\end{equation} where $[x^k] f(x)$ denotes the coefficient of $x^k$ in
a polynomial $f(x)$. If $i=r \mod b$, write $j=\frac{i-r}{b} +
\epsilon_{t+1}$. Then \eqref{54} becomes
\[ \frac{1}{b^n} [x^{b-r-1+(\epsilon_{t+1})b}]
(1+x+\cdots+x^{b-1})^{n+1}.\] To see that this implies the lemma, note
that the sum of $n+1$ discrete uniforms on $\{0,1,\cdots,b-1\}$ is
equidistributed mod $b$, and so is congruent to $b-r-1 \mod b$ with
probability $\frac{1}{b}$. \end{proof}

\begin{proof} (Of Theorem \ref{thm5.2}) We show convergence by showing
that $\{P_n\}_{n=1}^\infty$ is tight and that the finite dimensional
distributions of $P_n$ converge to the finite dimensional distributions of
$Q$. This is enough from \cite[2.2, 4.3, 4.5]{eth}. From \cite[2.4]{eth},
$P_n$ is tight if and only if the family $P_n^h$ of $h$th marginal
distributions is tight. Thus it is enough to show that the finite
dimensional distributions converge.

By Lemma \ref{rel} the carries process can be represented as:
\begin{equation} If \ \kappa_t = r \mod b, \ \ \ let \
\kappa_{t+1}=\frac{\kappa_t - r}{b} + \epsilon_{t+1} \end{equation} where
$\mathbb{P}(\epsilon_{t+1}=k)$ is the chance that the sum of $n+1$ i.i.d.
uniforms on $\{0,1,\cdots,b-1\}$ is equal to $bk+b-r-1$, given that the
sum is congruent to $b-r-1$ mod $b$. Hence $b \epsilon_{t+1}$ (for
$\kappa_t = b-1$ mod b) has the distribution of the sum of $n+1$ uniforms
on $\{0,1,\cdots,b-1\}$ given that the sum is congruent to $0$ mod b. A
generating function argument then shows that if $n \geq 2$ then
$\epsilon_{t+1}$ has mean $\frac{n+1}{2} \left( 1 - \frac{1}{b} \right)$
and variance $\frac{n+1}{12} \left( 1 - \frac{1}{b^2} \right)$. By the
local central limit theorem for sums of i.i.d. random variables,
\begin{equation*}
\frac{\epsilon_{t+1}- \frac{n}{2} \left( 1 - \frac{1}{b}
\right)}{\sqrt{\frac{n}{12} \left( 1 - \frac{1}{b^2} \right)}
}\Longrightarrow\mathcal{N}(0,1)
\end{equation*} as $n \rightarrow \infty$, and a similar argument gives
the same conclusion for $\kappa_t$ congruent to any $r$ mod b, with error
term $O(n^{-1/2})$ since $b$ is fixed.

From these considerations, the joint distribution of
$(\epsilon_1,\epsilon_2,\dots,\epsilon_h)$, normalized as above, converges
to the product of $h$ independent standard normal variables ($h$ fixed,
$n$ large).

Next, represent $\kappa_{t+1}=\frac{\kappa_t-\delta_t}{b}+\epsilon_{t+1}$
with $\delta_t= \kappa_t \mod b$. Thus, for $t=1,2,3,\dots,h-1$, with
$\kappa_0=\delta_0=0$,
\begin{equation*}
\kappa_{t+1}=\frac{\kappa_0}{b^{t+1}}+\left(\epsilon_{t+1}+\frac{\epsilon_t}{b}
+\dots+\frac{\epsilon_1}{b^t}\right) -
\left(\frac{\delta_t}{b}+\frac{\delta_{t-1}}{b^2}+\dots+\frac{\delta_0}{b^{t+1}}
\right). \end{equation*} Since $\kappa_{t+1}=\sqrt{n/12} \cdot Y_{t+1} +
\frac{n}{2}$ and $\kappa_0=\sqrt{n/12} \cdot Y_0 + \frac{n}{2}$, it
follows that
\begin{eqnarray*}
Y_{t+1} & = & \frac{Y_0}{b^{t+1}} - \frac{n/2}{\sqrt{n/12}}
\left(1-\frac{1}{b^{t+1}} \right) + \frac{1}{\sqrt{n/12}} \left(
\epsilon_{t+1}+\frac{\epsilon_t}{b} +\dots+\frac{\epsilon_1}{b^t}\right)\\
& & - \frac{1}{\sqrt{n/12}}
\left(\frac{\delta_t}{b}+\frac{\delta_{t-1}}{b^2}+\dots+\frac{\delta_0}{b^{t+1}}
\right) \\
& = & \frac{Y_0}{b^{t+1}} + \frac{1}{\sqrt{n/12}} \left[ \left(
\epsilon_{t+1}- \frac{n}{2}(1-\frac{1}{b}) \right)
+\dots+ \frac{\left(\epsilon_1 -\frac{n}{2} (1- \frac{1}{b}) \right)}{b^t} \right]\\
& & - \frac{1}{\sqrt{n/12}}
\left(\frac{\delta_t}{b}+\frac{\delta_{t-1}}{b^2}+\dots+\frac{\delta_0}{b^{t+1}}
\right). \end{eqnarray*} As noted earlier, the $\frac{\epsilon_i -
\frac{n}{2}(1-1/b)}{\sqrt{n/12}}$ converge to independent
$\mathcal{N}(0,1-\frac{1}{b^2})$'s. Since $|\delta_i| \leq b$ for all $i$,
the term involving the $\delta$'s converges to 0 with probability 1, so by
Slutsky's theorem, it can be disregarded. We thus have that the joint
distribution of $(Y_0,Y_1,\dots,Y_h)$ converges to the joint distribution
of $(W_0,W_1,\dots,W_h)$, and the proof is complete.
\end{proof}
\begin{rem}
  The Gaussian autoregressive process
  $X_{n+1}=\frac{1}{b}X_n+\epsilon_{n+1}$ (with $X_0=x$) is carefully
  studied in \cite{pd164}. It has eigenvalues $1,1/b,1/b^2,\dots$ and
  takes order $\log_b|x|+c$ steps to converge \cite[Prop.~4.9]{pd164}.
  Taking $x=-\sqrt{3n}$ as in Theorem \ref{thm5.2}, this is consistent
  with our result in \ref{sec3} that $\tfrac12\log_b(n)+c$ steps is the
  right answer for the carries chain.
\end{rem}
\begin{rem}
Theorem \ref{thm5.2} implies that many properties of Gaussian
autoregressive processes (here the discrete time Ornstein-Uhlenbeck
process) apply to carries--at least in the limit. For example Corollary 2
of Lai \cite{Lai} implies, in the notation above, that \[ \mathbb{P}(W_t
\geq b_t \ {\rm i.o.}) = 1 \ {\rm or} \ 0 \ {\rm according \ as} \
\sum_{t=0}^{\infty} b_t^{-1}e^{-b_t^2/2} = \infty \ {\rm or} \ < \infty.\]
It follows for carries that \[ \mathbb{P} \left( \kappa_t \geq \frac{n}{2}
+ \sqrt{\frac{n}{12} (\log(t))^{1+\epsilon}} \ {\rm i.o.} \right) = 1 \
{\rm or} \ 0 \ {\rm according \ as} \ \epsilon \leq 0 \ {\rm or} \
\epsilon
> 0.\]
\end{rem}

\section*{Acknowledgments} We thank Christos Athanasiadis, Sourav Chatterjee, and
Tze Lai for their help. Diaconis was partially supported by NSF grant DMS
0505673. Fulman was partially supported by NSF grant DMS 0802082 and NSA
grant H98230-08-1-0133.

\end{document}